\documentclass{amsart}
\usepackage{amsmath, amsfonts, amssymb, stmaryrd}
\usepackage{xypic}
\usepackage[colorlinks=true, linkcolor=blue,
  citecolor=blue, urlcolor=blue]{hyperref}

\def\newthm#1#2{\newtheorem{#1}[dummy]{#2}%
  \expandafter\def\csname#2\endcsname##1{\hyperref[#1:##1]{#2~\ref*{#1:##1}}}}
\newthm{thm}{Theorem}
\newthm{lemma}{Lemma}
\newthm{prop}{Proposition}
\newthm{cor}{Corollary}
\newthm{conj}{Conjecture}
\newthm{cons}{Consequence}
\newthm{fact}{Fact}
\newthm{obs}{Observation}
\newthm{guess}{Guess}
\theoremstyle{definition}
\newthm{defn}{Definition}
\newthm{remark}{Remark}
\newthm{example}{Example}
\newthm{question}{Question}
\newthm{notation}{Notation}

\newcommand{\Section}[1]{\hyperref[sec:#1]{Section~\ref*{sec:#1}}}
\newcommand{\Table}[1]{\hyperref[tab:#1]{Table~\ref*{tab:#1}}}
\newcommand{\eqn}[1]{\hyperref[eqn:#1]{(\ref*{eqn:#1})}}

\DeclareMathOperator{\Gr}{Gr}
\DeclareMathOperator{\Fl}{Fl}
\DeclareMathOperator{\LG}{LG}

\DeclareMathOperator{\OG}{OG}

\DeclareMathOperator{\GL}{GL}

\DeclareMathOperator{\QH}{QH}
\DeclareMathOperator{\QK}{QK}

\DeclareMathOperator{\codim}{codim}
\DeclareMathOperator{\Lie}{Lie}

\renewcommand{\P}{{\mathbb P}}
\newcommand{\C}{{\mathbb C}}

\newcommand{\Z}{{\mathbb Z}}

\newcommand{\cF}{{\mathcal F}}

\newcommand{\cO}{{\mathcal O}}

\newcommand{\cZ}{{\mathcal Z}}

\newcommand{\euler}[1]{\chi_{_{#1}}}
\newcommand{\Teuler}[1]{\chi_{_{#1}}^{T}}
\newcommand{\pt}{\text{pt}}

\newcommand{\la}{{\lambda}}

\newcommand{\ga}{{\gamma}}
\newcommand{\Bl}{\text{B$\ell$}}

\newcommand{\ev}{\operatorname{ev}}
\newcommand{\wt}{\widetilde}

\newcommand{\wb}{\overline}
\newcommand{\ov}{\overline}

\newcommand{\ignore}[1]{}

\newcommand{\Mb}{\wb{\mathcal M}}
\def\noin{\noindent}
\def\N{{\mathbb N}}

\DeclareMathOperator{\im}{Im}

\begin{document}

\title{Projected Gromov-Witten varieties in cominuscule spaces}

\date{June 7, 2017}

\author{Anders~S.~Buch}
\address{Department of Mathematics, Rutgers University, 110
  Frelinghuysen Road, Piscataway, NJ 08854, USA}
\email{asbuch@math.rutgers.edu}

\author{Pierre--Emmanuel Chaput}
\address{Domaine Scientifique Victor Grignard, 239, Boulevard des
  Aiguillettes, Universit{\'e} de Lorraine, B.P. 70239,
  F-54506 Vandoeuvre-l{\`e}s-Nancy Cedex, France}
\email{pierre-emmanuel.chaput@univ-lorraine.fr}

\author{Leonardo~C.~Mihalcea}
\address{Department of Mathematics, Virginia Tech University, 460 McBryde,
  Blacksburg VA 24060, USA}
\email{lmihalce@math.vt.edu}

\author{Nicolas Perrin}
\address{Laboratoire de Math{\'e}matiques de Versailles, UVSQ, CNRS,
  Universit{\'e} Paris-Saclay, 78035 Versailles, France}
\email{nicolas.perrin@uvsq.fr}


\thanks{The first author was supported in part by NSF grant
  DMS-1205351.}

\thanks{The third author was supported in part by NSA Awards
  H98230-13-1-0208 and H98320-16-1-0013 and a Simons Collaboration
  Grant.}

\thanks{The fourth author was supported by a public grant as part of
  the Investissement d'avenir project, reference ANR-11-LABX-0056-LMH,
  LabEx LMH.}

\begin{abstract}
  A projected Gromov-Witten variety is the union of all rational
  curves of fixed degree that meet two opposite Schubert varieties in
  a homogeneous space $X = G/P$.  When $X$ is cominuscule we prove
  that the map from a related Gromov-Witten variety is cohomologically
  trivial.  This implies that all (3 point, genus zero) $K$-theoretic
  Gromov-Witten invariants of $X$ are determined by projected
  Gromov-Witten varieties, which extends an earlier result of Knutson,
  Lam, and Speyer, and provides an alternative version of the `quantum
  equals classical' theorem.  Our proof uses that any projected
  Gromov-Witten variety in a cominuscule space is also a projected
  Richardson variety.
\end{abstract}

\maketitle

\markboth{A.~BUCH, P.--E.~CHAPUT, L.~MIHALCEA, AND N.~PERRIN}
{PROJECTED GROMOV-WITTEN VARIETIES IN COMINUSCULE SPACES}

\section{Introduction}\label{sec:intro}

Let $X = G/P$ be a homogeneous space defined by a simple complex Lie
group $G$ and a parabolic subgroup $P$.  Given an effective degree
$d \in H_2(X;\Z)$, the Kontsevich moduli space $M_d = \Mb_{0,3}(X,d)$
parametrizes the set of 3-pointed stable maps to $X$ of degree $d$ and
genus zero.  Let $\ev_i : M_d \to X$ be the $i$-th evaluation map.
Given two opposite Schubert varieties $X_u$ and $X^v$ in $X$, let
$M_d(X_u,X^v) = \ev_1^{-1}(X_u) \cap \ev_2^{-1}(X^v) \subset M_d$
denote the {\em Gromov-Witten variety\/} of stable maps that send the
first marked point to $X_u$ and the second marked point to $X^v$.  It
was proved in \cite{buch.chaput.ea:finiteness} that $M_d(X_u,X^v)$ is
either empty or unirational with rational singularities.  The image
$\Gamma_d(X_u,X^v) = \ev_3(M_d(X_u,X^v)) \subset X$ is called a {\em
  projected Gromov-Witten variety}.  This variety $\Gamma_d(X_u,X^v)$
is the union of all connected rational curves of degree $d$ in $X$
that meet both of the Schubert varieties $X_u$ and $X^v$.

Gromov-Witten varieties are closely related to the (ordinary and
$K$-theoretic) Gro\-mov-Witten invariants.  These invariants are
defined by
\[
I_d([\cO_{X_u}], [\cO_{X^v}], \sigma ) = \euler{M_d}(
\ev_1^*[\cO_{X_u}] \cdot \ev_2^*[\cO_{X^v}] \cdot \ev_3^*(\sigma)) \,,
\]
where $\sigma \in K(X)$ is any $K$-theory class and
$\euler{M_d}: K(M_d) \to \Z$ is the sheaf Euler characteristic map.
It follows from Sierra's sheaf-theoretic version of Kleiman's
transversality theorem \cite{sierra:general} that
$\ev_1^* [\cO_{X_u}] \cdot \ev_2^*[\cO_{X^v}] = [\cO_{M_d(X_u, X^v)}]
\in K(M_d)$,
see \cite[\S 4]{buch.mihalcea:quantum}.  In particular, any 2-point
Gromov-Witten invariant of $X$ is determined by
$I_d([\cO_{X_u}], [\cO_{X^v}], 1) =
\euler{M_d}([\cO_{M_d(X_u,X^v)}])$,
and so is equal to zero or one by the geometric properties of
Gromov-Witten varieties proved in \cite{buch.chaput.ea:finiteness}.
More generally, the projection formula implies that the 3-point
Gromov-Witten invariants of $X$ are given by
\[
I_d([\cO_{X_u}], [\cO_{X^v}], \sigma )= \euler{X}
((\ev_3)_*[\cO_{M_d(X_u, X^v)}] \cdot \sigma) \,.
\]
It is therefore natural to study the cohomological properties of the
restricted evaluation map
\begin{equation}\label{eqn:res}
  \ev_3 : M_d(X_u,X^v) \to \Gamma_d(X_u,X^v) \,.
\end{equation}

In this paper we study the projected Gromov-Witten varieties when $X$
is a cominuscule variety, i.e.\ a Grassmannian of type A, a Lagrangian
Grassmannian, a maximal orthogonal Grassmannian, a quadric
hypersurface, or one of two exceptional varieties called the Cayley
plane and the Freudenthal variety.  Our main result states that, when
$X$ is cominuscule, the restricted evaluation map \eqn{res} is
{\em cohomologically trivial}, which means that the pushforward of the
structure sheaf of $M_d(X_u,X^v)$ is the structure sheaf of
$\Gamma_d(X_u,X^v)$ and the higher direct images of the first sheaf
are zero.  This result implies that the (small) quantum cohomology
ring $\QH(X)$ and the quantum $K$-theory ring $\QK(X)$ are determined
by the projected Gromov-Witten varieties in $X$.  More precisely, the
$K$-theoretic (3-point genus zero) Gromov-Witten invariants of $X$
satisfy the identity
\begin{equation}\label{eqn:gwinv}
  I_d([\cO_{X_u}], [\cO_{X^v}], \sigma) =
  \euler{X}([\cO_{\Gamma_d(X_u,X^v)}] \cdot \sigma) \,.
\end{equation}
Knutson, Lam, and Speyer have earlier proved this identity for
cohomological Gromov-Witten invariants of Grassmannians of type A
\cite{knutson.lam.ea:positroid}.  We apply our identity to compute the
square of a point in the quantum $K$-theory ring of any cominuscule
variety.

The identity \eqn{gwinv} can be interpreted as an alternative version
of the `quantum equals classical' theorem for cominuscule
Gromov-Witten invariants, which was proved in various generalities in
the papers \cite{buch.kresch.ea:gromov-witten,
  chaput.manivel.ea:quantum*1, buch.mihalcea:quantum,
  chaput.perrin:rationality}.  It has the advantage of being
completely uniform.  In particular, it avoids a special case of the
original version that concerns the degree $d=3$ for the Cayley plane
$E_6/P_6$ \cite{chaput.perrin:rationality}.  In addition it reveals
that certain (equivariant) $K$-theoretic Gromov-Witten invariants have
alternating signs (see \Corollary{pos} and \Remark{pos}), generalizing
results of Buch \cite{buch:littlewood-richardson}, Brion
\cite{brion:positivity}, and Anderson, Griffeth and Miller
\cite{anderson.griffeth.ea:positivity}.

If $d=0$, the Gromov-Witten variety $M_d(X_u, X^v)$ is an intersection
$X_u \cap X^v$ of opposite Schubert varieties in $X$, also called a
{\em Richardson variety}.  Let $B \subset P$ be a Borel subgroup, set
$F = G/B$, and let $\rho : F \to X$ be the projection.  If
$R \subset F$ is any Richardson variety, then the image
$\rho(R) \subset X$ is called a {\em projected Richardson variety}.
Projected Richardson varieties have been studied by Lusztig
\cite{lusztig:total} and Rietsch \cite{rietsch:closure} in the context
of total positivity.  For arbitrary homogeneous spaces $X = G/P$ it
was proved by Billey and Coskun \cite{billey.coskun:singularities} and
by Knutson, Lam, and Speyer \cite{knutson.lam.ea:projections} that any
projected Richardson variety $\rho(R)$ is Cohen-Macaulay with rational
singularities, and the restricted map $\rho : R \to \rho(R)$ is
cohomologically trivial.  He and Lam have recently related the
$K$-theory classes of projected Richardson varieties to the
$K$-homology of the affine Grassmannian \cite{he.lam:projected}.

The results of this paper show that the restricted evaluation map
$\ev_3$ and the map $\rho$ have similar properties when $X$ is a
cominuscule variety.  We already mentioned cohomological triviality.
In addition, if $X$ is cominuscule, then any projected Gromov-Witten
variety in $X$ is also a projected Richardson variety.  This result is
derived from the main construction used to prove the $K$-theoretic
version of the `quantum equals classical' theorem in
\cite{buch.mihalcea:quantum, chaput.perrin:rationality}.  In type A
this analysis shows that any variety $\Gamma_d(X_u,X^v)$ is the image
of a Richardson variety in a three-step flag manifold obtained from
the Richardson variety of kernel-span pairs of rational curves passing
through $X_u$ and $X^v$; see \S 5, and also \cite{buch:quantum,
  buch.kresch.ea:gromov-witten, knutson.lam.ea:positroid}.  In most
cases we can derive the cohomological triviality of the map
$\ev_3 : M_d(X_u,X^v) \to \Gamma_d(X_u,X^v)$ from the cohomological
triviality of a projection $\rho : R \to \rho(R) = \Gamma_d(X_u,X^v)$
of a Richardson variety $R$.  However, a separate argument is needed
when $X$ is the Cayley plane $E_6/P_6$ and $d=3$, as in this case the
analogue of the three-step flag manifold constructed in
\cite{chaput.perrin:rationality} fails to be a homogeneous space.  In
this case we obtain our result by proving that the general fibers of
the restricted evaluation map are rationally connected.

This paper is organized as follows.  In \Section{ctriv} we briefly
discuss cohomologically trivial maps and state some useful results.
\Section{intschub} explains our notation for Schubert varieties and
proves two results about intersections of Schubert varieties in
special position.  In \Section{projgw} we define projected
Gromov-Witten varieties and state our main result as well as some
consequences.  \Section{qclas} recalls the `quantum equals classical'
theorem for cominuscule varieties and uses it to prove our main result
whenever $X$ is not the Cayley plane or $d \neq 3$.  This section also
explains how to compute the dimension of a projected Gromov-Witten
variety and gives an example of a projected Richardson variety in a
cominuscule space that is not a projected Gromov-Witten variety.
Finally, \Section{cayley} proves the main theorem for the Cayley
plane, and \Section{square} gives the formula for the square of a
point in the quantum $K$-theory ring.

{\em Acknowledgment.} We thank Thomas Lam for useful discussions
regarding the dimension of a projected Richardson variety.  We also
thank the referee for several helpful comments.

\section{Cohomologically trivial maps}\label{sec:ctriv}

In this section we state some facts about cohomologically trivial maps
which are required in later sections.

\begin{defn}
  A morphism $f : X \to Y$ of schemes is {\em cohomologically
    trivial\/} if we have $f_* \cO_X = \cO_Y$ and $R^i f_* \cO_X = 0$
  for $i > 0$.
\end{defn}

Notice that if $f : X \to Y$ is proper and cohomologically trivial,
then we have $f_* [\cO_X] = [\cO_Y] \in K_\circ(Y)$ in the
Grothendieck group of coherent sheaves on $Y$.  All the
cohomologically trivial maps encountered in this paper are also
proper.  It would be interesting to know if cohomological triviality
implies properness.

An irreducible complex variety $X$ has {\em rational singularities\/}
if there exists a cohomologically trivial resolution of singularities
$\pi : \wt X \to X$, i.e.\ $\wt X$ is a non-singular variety and $\pi$
is a proper birational and cohomologically trivial morphism.  If $X$
has rational singularities, then $X$ is normal, and all resolutions of
singularities of $X$ are cohomologically trivial.

An irreducible variety $X$ is {\em rationally connected\/} if a
general pair of points $(x,y) \in X \times X$ can be joined by a
rational curve, i.e.\ both $x$ and $y$ belong to the image of some
morphism $\P^1 \to X$.  The following result provides a sufficient
condition for cohomological triviality.  It was proved in
\cite[Thm.~3.1]{buch.mihalcea:quantum} as an application of a theorem
of Koll\'ar \cite{kollar:higher}, see also
\cite[Prop.~5.2]{buch.chaput.ea:finiteness}.

\begin{prop}\label{prop:gysin}
  Let $f: X \to Y$ be a surjective morphism between complex projective
  varieties with rational singularities.  Assume that the general
  fibers of $f$ are rationally connected.  Then $f$ is cohomologically
  trivial.
\end{prop}

To establish that a variety is rationally connected, the following
result from \cite{graber.harris.ea:families} is invaluable.

\begin{prop}[Graber, Harris, Starr]\label{prop:ratconn}
  Let $f : X \to Y$ be any dominant morphism of complete irreducible
  complex varieties.  If $Y$ and the general fibers of $f$ are
  rationally connected, then $X$ is rationally connected.
\end{prop}

We need the following statement about cohomological triviality of
compositions, which will be applied in both directions.

\begin{lemma}\label{lemma:composite}
  Let $f : X \to Y$ be a cohomologically trivial morphism, and let
  $g : Y \to Z$ be any morphism of schemes.  Then $g$ is
  cohomologically trivial if and only if $gf$ is cohomologically
  trivial.
\end{lemma}
\begin{proof}
  The Grothendieck spectral sequence shows that
  $R^i g_* \cO_Y = R^i (gf)_* \cO_X$ for all $i \geq 0$.
\end{proof}

\section{Intersections of Schubert varieties}\label{sec:intschub}

In this section we fix our notation for Schubert varieties and state
some results.  Let $X = G/P$ be a homogeneous space defined by a
reductive complex linear algebraic group $G$ and a parabolic subgroup
$P$.  An irreducible closed subvariety $\Omega \subset X$ is called a
{\em Schubert variety\/} if there exists a Borel subgroup
$B \subset G$ such that $\Omega$ is $B$-stable, i.e.\
$B.\Omega = \Omega$.

Let $Q \subset P$ be a parabolic subgroup contained in $P$ and
consider the projection $\pi : G/Q \to G/P$.  We need the following
result.

\begin{prop}\label{prop:projschub}
  Let $\Omega$ be any Schubert variety in $G/Q$.  Then for any
  $x \in G/P$, the fiber $\pi^{-1}(x)$ is a homogeneous space for a
  conjugate of a Levi subgroup of $P$.  Furthermore, the intersection
  $\Omega \cap \pi^{-1}(x)$ is a Schubert variety in $\pi^{-1}(x)$ for
  all points $x$ in a dense open subset of $\pi(\Omega)$.
\end{prop}
\begin{proof}
  We may assume that $\Omega = \ov{B.Q}$ and $\pi(\Omega) = \ov{B.P}$
  for some Borel subgroup $B \subset G$.  Since the dense open orbit
  $B.P \subset \pi(\Omega)$ is a principal homogeneous space for a
  subgroup of $B$, it follows that the restricted map
  $\pi : \Omega \to \pi(\Omega)$ can be trivialized over this orbit,
  i.e.\ there is an isomorphism
  $\pi^{-1}(B.P) \cap \Omega \cong B.P \times F$ with
  $F = \pi^{-1}(1.P) \cap \Omega$, such that the map
  $\pi^{-1}(B.P) \cap \Omega \to B.P$ is the projection to the first
  factor (see also \cite[Prop.~2.3]{buch.chaput.ea:finiteness}).  It
  follows that $F$ is irreducible, and it suffices to show that $F$ is
  a Schubert variety in $\pi^{-1}(1.P) = P/Q$.  Choose a maximal torus
  $T$ in $G$ such that $T \subset B \cap Q$, and let $P = LU$ be the
  Levi decomposition of $P$ with respect to $T$, i.e.\ $L$ is a
  (reductive) Levi subgroup containing $T$ and $U$ is the unipotent
  radical.  Since $U \subset Q$ it follows that
  $\pi^{-1}(1.P) = L/(L\cap Q)$ is a homogeneous space for $L$.  By
  using the root space decomposition of $\Lie(G)$ it follows that
  $B' = B \cap L$ is a Borel subgroup of $L$.  The proposition now
  follows because $F$ is a $B'$-stable closed irreducible subvariety
  of $\pi^{-1}(1.P)$.
\end{proof}

From now on we will fix a Borel subgroup $B$ and a maximal torus $T$
such that $T \subset B \subset P \subset G$.  Let $W = N_G(T)/T$ be
the Weyl group of $G$, let $W_P = N_P(T)/T \subset W$ be the Weyl
group of $P$, and let $W^P \subset W$ be the subset of minimal
representatives for the cosets in $W/W_P$.  Each element $u \in W$
defines a Schubert variety $X_u = \overline{Bu.P}$ and an opposite
Schubert variety $X^u = \overline{B^-u.P}$ in $X$.  Here
$B^- \subset G$ is the Borel subgroup opposite to $B$.  For
$u \in W^P$ we have $\dim(X_u) = \codim(X^u,X) = \ell(u)$, where
$\ell(u)$ denotes the length of $u$.  Any non-empty intersection of
the form $X_u \cap X^v$ with $u,v \in W$ is called a {\em Richardson
  variety}.  Richardson varieties are known to be rational
\cite{richardson:intersections}, and they have rational singularities
\cite{brion:positivity}.

\begin{prop}\label{prop:degrich}
  Let $u,v \in W^P$ be such that $X_u \cap X^v \neq \emptyset$.  Then
  $X_u \cap g.X^v$ is connected for all $g \in G$.
\end{prop}
\begin{proof}
  We follow Brion's proof of \cite[Lemma~2]{brion:positivity}.  Let
  $G$ act on $G \times X^v$ by left multiplication on the first
  factor.  Then the multiplication map $m : G \times X^v \to X$
  defined by $m(g,x) = g.x$ is $G$-equivariant.  It follows that $m$
  is a locally trivial fibration, see
  \cite[Prop.~2.3]{buch.chaput.ea:finiteness}.  We deduce that
  $Z = m^{-1}(X_u) = X_u \times_X (G \times X^v)$ is an irreducible
  variety.  Let $\pi : Z \to G$ be the projection.  For $g \in G$ we
  then have $\pi^{-1}(g) = X_u \cap g.X^v$.  Since this is a translate
  of a Richardson variety for all elements $g$ in a dense open subset
  of $G$, it follows that the general fibers of $\pi$ are connected.
  By Zariski's Main Theorem and Stein Factorization, this implies that
  all fibers of $\pi$ are connected, as required.
\end{proof}

Set $F = G/B$ and let $\rho : F \to X$ be the projection.  If $R$ is
any Richardson variety in $F$, then the image $\rho(R)$ is called a
{\em projected Richardson variety}.  We need the following result
which was proved in the papers \cite{billey.coskun:singularities,
  knutson.lam.ea:projections}.

\begin{prop}[\cite{billey.coskun:singularities,
    knutson.lam.ea:projections}]\label{prop:projrich}
  Let $R \subset G/B$ be a Richardson variety and let
  $\rho(R) \subset G/P$ be the corresponding projected Richardson
  variety.\smallskip

  \noin{\rm(a)} The variety $\rho(R)$ is Cohen-Macaulay and has
  rational singularities.\smallskip

  \noin{\rm(b)} The restricted map $\rho : R \to \rho(R)$ is
  cohomologically trivial.
\end{prop}

\section{Projected Gromov-Witten varieties}\label{sec:projgw}

Let $\Phi$ be the root system of $(G,T)$, with positive roots $\Phi^+$
and simple roots $\Delta \subset \Phi^+$.  In the rest of this paper
we will assume that $X = G/P$ is a {\em cominuscule variety}.  This
means that $P$ is a maximal parabolic subgroup of $G$ corresponding to
a simple root $\ga \in \Delta$ such that $s_\ga \notin W_P$.  In
addition $\ga$ is a cominuscule simple root, i.e.\ when the highest
root in $\Phi^+$ is written as a linear combination of simple roots,
the coefficient of $\ga$ is one.  The collection of cominuscule
varieties consists of Grassmannians $\Gr(m,N)$ of type A, Lagrangian
Grassmannians $\LG(m,2m)$, maximal orthogonal Grassmannians
$\OG(m,2m)$, quadric hypersurfaces $Q^m$, and two exceptional
varieties called the Cayley plane $E_6/P_6$ and the Freudenthal
variety $E_7/P_7$.

We will identify the homology group $H_2(X;\Z)$ with the integers
$\Z$, so that the generator $[X_{s_\ga}] \in H_2(X;\Z)$ corresponds to
$1 \in \Z$.  Given a non-negative degree $d \in \Z$ and a positive
integer $n$, the Kontsevich moduli space $\Mb_{0,n}(X,d)$ parametrizes
the isomorphism classes of $n$-pointed stable (genus zero) maps
$f : C \to X$ for which $f_*[C] = d$, and comes with an evaluation map
$\ev = (\ev_1,\dots,\ev_n) : \Mb_{0,n}(X,d) \to X^n := X \times \dots
\times X$.
A detailed construction of this space can be found in the survey
\cite{fulton.pandharipande:notes}.  For the special case $n=3$ we
write simply $M_d = \Mb_{0,3}(X,d)$.  Given two closed subvarieties
$\Omega_1$ and $\Omega_2$ in $X$, define the {\em Gromov-Witten
  variety\/}
$M_d(\Omega_1, \Omega_2) = \ev_1^{-1}(\Omega_1) \cap
\ev_2^{-1}(\Omega_2) \subset M_d$.
The corresponding {\em projected Gromov-Witten variety\/} is defined
by
$\Gamma_d(\Omega_1,\Omega_2) = \ev_3(M_d(\Omega_1,\Omega_2)) \subset
X$.
It was proved in \cite{buch.chaput.ea:finiteness} that any
Gromov-Witten variety $M_d(X_u,X^v)$ defined by two opposite Schubert
varieties is either empty or unirational with rational singularities.
Our main result is the following theorem.

\begin{thm}\label{thm:main}
  Let $X$ be a cominuscule variety and let $X_u$ and $X^v$ be opposite
  Schubert varieties in $X$ such that
  $\Gamma_d(X_u,X^v) \neq \emptyset$.\smallskip

  \noin{\rm(a)} The projected Gromov-Witten variety
  $\Gamma_d(X_u,X^v)$ is a projected Richardson variety.  In
  particular, it is Cohen-Macaulay and has rational
  singularities.\smallskip

  \noin{\rm(b)} The restricted map
  $\ev_3 : M_d(X_u,X^v) \to \Gamma_d(X_u,X^v)$ is cohomologically
  trivial.
\end{thm}

Let $K(X)$ denote the Grothendieck ring of algebraic vector bundles on
$X$.  We use the notation $\cO_u = [\cO_{X_u}]$ and
$\cO^u = [\cO_{X^u}]$ for the classes in $K(X)$ defined by the
structure sheaves of Schubert varieties.  The set
$\{ \cO_u \mid u \in W^P \} = \{\cO^u \mid u \in W^P \}$ is a
$\Z$-basis for $K(X)$.  Let $\cO_w^\vee$ denote the basis element dual
to $\cO^w$, in the sense that
$\euler{X}(\cO^u \cdot \cO_w^\vee) = \delta_{u,w}$ for all
$u,w \in W^P$.  Here $\euler{X} : K(X) \to \Z$ denotes the sheaf Euler
characteristic defined by
$\euler{X}(\cF) = \sum_{k \geq 0} (-1)^k \dim H^i(X;\cF)$.  Given
three classes $\sigma_1, \sigma_2, \sigma_3 \in K(X)$, define a
$K$-theoretic Gromov-Witten invariant by
\[
I_d(\sigma_1,\sigma_2,\sigma_3) = \euler{M_d}(\ev_1^*(\sigma_1) \cdot
\ev_2^*(\sigma_2) \cdot \ev_3^*(\sigma_3)) \,.
\]
\Theorem{main} has the following consequence.

\begin{cor}\label{cor:pgw_class}
  In the ring $K(X)$ we have
  \[
  [\cO_{\Gamma_d(X_u,X^v)}] = (\ev_3)_* [\cO_{M_d(X_u,X^v)}] =
  \sum_{w\in W^P} I_d(\cO_u,\cO^v,\cO_w^\vee)\, \cO^w \,.
  \]
\end{cor}

We note that the second equality in this corollary is clear from the
definition of $K$-theoretic Gromov-Witten varieties.  In fact, we have
$[\cO_{M_d(X_u,X^v)}] = \ev_1^*(\cO_u) \cdot \ev_2^*(\cO^v)$ in
$K(M_d)$ by \cite[Thm.~2.2]{sierra:general} (see
\cite[\S4.1]{buch.mihalcea:quantum}), from which we deduce that
\[
\euler{X}((\ev_3)_*[\cO_{M_d(X_u,X^v)}] \cdot \cO_w^\vee) =
\euler{M_d}([\cO_{M_d(X_u,X^v)}] \cdot \ev_3^*(\cO_w^\vee)) =
I_d(\cO_u,\cO^v,\cO_w^\vee) \,,
\]
as required.

A theorem of Brion \cite{brion:positivity} states that, if a closed
irreducible subvariety of a homogeneous space has rational
singularities, then the expansion of its Grothendieck class in the
basis of Schubert structure sheaves has alternating signs.  This
combined with \Corollary{pgw_class} has the following consequence.

\begin{cor}\label{cor:pos}
  The $K$-theoretic Gromov-Witten invariants
  $I_d(\cO_u,\cO^v,\cO_w^\vee)$ have alternating signs in the sense
  that
  \[
  (-1)^{\ell(w)-\codim \Gamma_d(X_u,X^v)} \,
  I_d(\cO_u,\cO^v,\cO_w^\vee) \geq 0 \,.
  \]
\end{cor}

\noindent
We will explain how to compute the codimension of $\Gamma_d(X_u,X^v)$
in \Remark{codim}.

\begin{remark}\label{remark:pos}
  Since all relevant maps and classes are $T$-equivariant,
  \Corollary{pgw_class} holds more generally for the class of
  $\Gamma_d(X_u,X^v)$ in the Grothendieck ring $K_T(X)$ of
  $T$-equivariant vector bundles on $X$, with the same proof.  To be
  precise, let $\cO_u = [\cO_{X_u}]$ and $\cO^v = [\cO_{X^v}]$ denote
  $T$-equivariant Schubert classes in $K_T(X)$, define
  $\cO_w^\vee \in K_T(X)$ by
  $\Teuler{X}(\cO^u \cdot \cO_w^\vee) = \delta_{u,w}$ where
  $\Teuler{X} : K_T(X) \to K_T(\pt)$ is the pushforward map along the
  structure morphism $X \to \{\pt\}$, and define $T$-equivariant
  $K$-theoretic Gromov-Witten invariants of $X$ by
  \[
  I^T_d(\cO_u, \cO^v, \cO_w^\vee) = \Teuler{M_d}(\ev_1^*(\cO_u) \cdot
  \ev_2^*(\cO^v) \cdot \ev_3^*(\cO_w^\vee)) \ \in K_T(\pt) \,.
  \]
  Then the equivariant Grothendieck class of $\Gamma_d(X_u,X^v)$ is
  given by
  \[
  [\cO_{\Gamma_d(X_u,X^v)}] = \sum_{w\in W^P} I^T_d(\cO_u, \cO^v,
  \cO_w^\vee)\, \cO^w \ \in K_T(X) \,.
  \]
  Furthermore, a generalization of Brion's theorem by Anderson,
  Griffeth, and Miller
  \cite[Cor.~5.1]{anderson.griffeth.ea:positivity} implies that the
  equivariant Gromov-Witten invariants of $X$ satisfy the positivity
  property
  \[
  (-1)^{\ell(w) - \codim \Gamma_d(X_u,X^v)}\, I^T_d(\cO_u, \cO^v,
  \cO_w^\vee) \ \in \ \N\big[\, [\C_\beta] - 1 : \beta \in \Delta\,
  \big] \,.
  \]
  In other words, up to a sign the invariant
  $I^T_d(\cO_u, \cO^v, \cO_w^\vee)$ can be written as a polynomial
  with non-negative integer coefficients in the classes
  $[\C_\beta] - 1 \in K_T(\pt)$ defined by the simple roots
  $\beta \in \Delta$.  Here $\C_\beta$ denotes the one-dimensional
  representation of $T$ defined by $t.z = \beta(t)z$ for $t \in T$ and
  $z \in \C$.
\end{remark}

\section{The quantum equals classical theorem}\label{sec:qclas}

In order to prove \Theorem{main} we need the `quantum equals
classical' theorem for Gromov-Witten invariants of cominuscule
varieties, which can be found in various generalities in the papers
\cite{buch.kresch.ea:gromov-witten, chaput.manivel.ea:quantum*1,
  buch.mihalcea:quantum, chaput.perrin:rationality}.  We will say that
a non-negative degree $d \in H_2(X)$ is {\em well behaved\/} if $X$ is
not the Cayley plane $E_6/P_6$ or $d \neq 3$.  We will explain the
theorem only for Gromov-Witten invariants of well behaved degrees.
The correct statement for Gromov-Witten invariants of the Cayley plane
of degree 3 can be found in \cite{chaput.perrin:rationality}.

For any non-negative degree $d$ and $n \in \N$ we set
$\cZ_{d,n} = \ev(\Mb_{0,n}(X,d)) \subset X^n$.  Given two points
$x,y \in X$, we let $d(x,y)$ denote the smallest degree of a rational
curve containing $x$ and $y$ \cite{zak:tangents}.  Equivalently,
$d(x,y)$ is the minimal degree $d$ for which $(x,y) \in \cZ_{d,2}$.
For $n \in \N$ we also let $d_X(n)$ be the smallest degree for which
any collection of $n$ points in $X$ is contained in a connected
rational curve of degree $d_X(n)$, i.e.\ $d_X(n)$ is minimal with the
property that $\cZ_{d_X(n),n} = X^n$.  The numbers $d_X(2)$ and
$d_X(3)$ are given in the following table (see
\cite[Prop.~18]{chaput.manivel.ea:quantum*1},
\cite[Prop.~3.4]{chaput.perrin:rationality}, and \cite[\S
4]{buch.chaput.ea:finiteness}).\medskip

\begin{center}
\begin{tabular}{c c c c c c}
  \hline\vspace{-3mm}\\
  $X$ & $\dim(X)$ & $d_X(2)$ & $d_X(3)$\vspace{1mm}\\
  \hline\vspace{-3mm}\\
  $\Gr(m,m+k)$ & $mk$ & $\min(m,k)$ & $\min(2m,2k,\max(m,k))$\vspace{1mm}\\
  $\LG(m,2m)$ & $\frac{m(m+1)}{2}$ & $m$ & $m$\vspace{1mm}\\
  $\OG(m,2m)$ & $\frac{m(m-1)}{2}$ & $\lfloor\frac{m}{2}\rfloor$ &
  $\lceil\frac{m}{2}\rceil$\vspace{1mm}\\
  $Q^m$ & $m$ & $2$ & $2$\vspace{1mm}\\
  $E_6/P_6$ & 18 & 2 & 4\vspace{1mm}\\
  $E_7/P_7$ & 27 & 3 & 3\vspace{1mm}\\
  \hline
\end{tabular}
\end{center}
\bigskip

Fix a well behaved degree $d \in H_2(X)$.  The main ingredient in the
`quantum equals classical' theorem is a homogeneous space $Y_d$ that
parametrizes a family of subvarieties in $X$.  If
$X = \Gr(m,N) = \{ V \subset \C^N \mid \dim(V)=m \}$ is a Grassmann
variety of type A, then $Y_d$ is the two-step partial flag variety
defined by
\[
Y_d = \Fl(a,b;N) = \{ (A,B) \mid A \subset B \subset \C^N ,\,
\dim(A)=a \,,\, \dim(B)=b \}
\]
where $a = \max(m-d,0)$ and $b = \min(m+d,N)$.  For each point
$\omega=(A,B) \in Y_d$ we set
$X_\omega = \Gr(m-a,B/A) = \{ V \in X \mid A \subset V \subset B \}$.
The idea is that the kernel $A$ and span $B$ \cite{buch:quantum} of a
general rational curve $C \subset X$ of degree $d$ form a point
$\omega = (A,B) \in Y_d$ such that $C \subset X_\omega$.

If $X$ is not a Grassmannian of type A, then we have either
$d \leq d_X(2)$ or $d \geq d_X(3)$.  In these cases we use the
set-theoretic definition
\[
Y_d = \{ \Gamma_d(x,y) \subset X \mid x,y \in X \text{ and } d(x,y) =
\min(d,d_X(2)) \} \,.
\]
For each element $\omega \in Y_d$ we let $X_\omega$ denote the
corresponding subvariety of $X$ of the form $\Gamma_d(x,y)$.  It
follows from \cite[Prop.~18]{chaput.manivel.ea:quantum*1} that $Y_d$
can be identified with a projective homogeneous space for $G$, and
that $X_\omega$ is a non-singular Schubert variety in $X$ for each
point $\omega \in Y_d$.  Notice that for $d \geq d_X(3)$, the variety
$Y_d$ is a single point and $X_\omega=X$ for $\omega \in Y_d$.

Define the incidence variety
$Z_d = \{ (\omega,x) \in Y_d \times X \mid x \in X_\omega \}$, and let
$p : Z_d \to X$ and $q : Z_d \to Y_d$ be the two projections.  For
$X = \Gr(m,N)$ we have $Z_d = \Fl(a,m,b;N)$ where $a$ and $b$ are
defined as above.  Otherwise it follows from
\cite[Prop.~18]{chaput.manivel.ea:quantum*1} that the stabilizer of
$X_\omega$ in $G$ acts transitively on $X_\omega$, which implies that
$Z_d$ is a projective homogeneous space for $G$ \footnote{The variety
  $Z_d$ is called $I_d$ in \cite{chaput.manivel.ea:quantum*1}, while
  $Y_d$ is called $F_d$, and $X_\omega$ is called $Y_d$.}.  The
following result holds more generally for the $T$-equivariant
$K$-theoretic Gromov-Witten invariants of $X$.

\begin{thm}[\cite{buch.kresch.ea:gromov-witten,
    chaput.manivel.ea:quantum*1, buch.mihalcea:quantum,
    chaput.perrin:rationality}]\label{thm:qclas}
  Let $X$ be a cominuscule variety and $d \in H_2(X)$ a well behaved
  degree.  Given classes $\sigma_1, \sigma_2, \sigma_3 \in K(X)$, the
  corresponding $K$-theoretic Gromov-Witten invariant of $X$ of degree
  $d$ is given by
  \[
  I_d(\sigma_1, \sigma_2, \sigma_3) = \euler{Y_d}(q_*p^*(\sigma_1)
  \cdot q_*p^*(\sigma_2) \cdot q_*p^*(\sigma_3)) \,.
  \]
\end{thm}

Define the variety
$\Bl_d = \{(\omega,f) \in Y_d \times M_d \mid \im(f) \subset X_\omega
\}$
and let $\pi : \Bl_d \to M_d$ be the projection.  This map is
birational by \cite[Lemma~4.6]{buch.mihalcea:quantum} and
\cite[Prop.~19]{chaput.manivel.ea:quantum*1}.  We also define
$Z_d^{(3)} = \{(\omega,x_1,x_2,x_3) \in Y_d \times X^3 \mid x_i \in
X_\omega \text{ for } 1 \leq i \leq 3 \}$,
and let $\phi : \Bl_d \to Z_d^{(3)}$ be the morphism defined by
$\phi(\omega,f) = (\omega,\ev(f))$.  It follows from
\cite[Cor.~2.2]{buch.mihalcea:quantum} and
\cite[Thm.~0.2]{chaput.perrin:rationality} that the general fibers of
$\phi$ are rational.  For $1 \leq i \leq 3$ we also define
$e_i : Z_d^{(3)} \to Z_d$ by $e_i(\omega,x_1,x_2,x_3) = (\omega,x_i)$.
We then have the following commutative diagram from
\cite{buch.mihalcea:quantum}.

\begin{equation}\label{eqn:diagram}
  \xymatrix{\Bl_d \ar[rr]^{\pi} \ar[d]^{\phi} & & M_d \ar[d]^{\ev_i} \\
    Z_d^{(3)}\ar[r]^{e_i} & Z_d \ar[r]^{p} \ar[d]^{q} & X \\ & Y_d}
\end{equation}

For $u, v \in W^P$ we define a Richardson variety in $Z_d$ by
\[
Z_d(X_u,X^v) = q^{-1}( q(p^{-1}(X_u)) \cap q(p^{-1}(X^v)) ) \,.
\]
Since both maps $\pi$ and $\phi$ are surjective, it follows from the
definitions that the projected Gromov-Witten variety for $u$ and $v$
is given by
\begin{equation}\label{eqn:gwprojrich}
  \Gamma_d(X_u,X^v) = p(Z_d(X_u,X^v)) \,.
\end{equation}
In particular, the projected Gromov-Witten variety $\Gamma_d(X_u,X^v)$
is also a projected Richardson variety.  This proves
\Theorem{main}(a).  We pause here to give a direct proof of
\Corollary{pgw_class} from \Theorem{qclas}, without invoking
\Theorem{main}(b).

\begin{proof}[Direct proof of \Corollary{pgw_class}]
  Assume first that $d$ is a well behaved degree.  For $u,v,w \in W^P$
  we then obtain from \Proposition{projrich} and \Theorem{qclas} that
  \[
  \begin{split}
    &\euler{X}([\cO_{\Gamma_d(X_u,X^v)}] \cdot \cO_w^\vee) =
    \euler{X}(p_*[\cO_{Z_d(X_u,X^v)}] \cdot \cO_w^\vee) =
    \euler{Z_d}([\cO_{Z_d(X_u,X^v)}] \cdot p^*(\cO_w^\vee)) \\
    & \ \ \ = \euler{Y_d}(q_*p^*(\cO_u) \cdot q_*p^*(\cO^v) \cdot
    q_*p^*(\cO_w^\vee)) = I_d(\cO_u,\cO^v,\cO_w^\vee) \,,
  \end{split}
  \]
  as required.

  Assume next that $X = E_6/P_6$ is the Cayley plane and $d=3$.  It
  then follows from \cite[Cor.~4.6]{buch.chaput.ea:finiteness} that
  $\Gamma_3(1.P, w_0.P)$ is a translate of the Schubert divisor
  $X^{s_\ga}$ in $X$, where $w_0 \in W$ is the longest element.  This
  in turn implies that, if either $X_u$ or $X^v$ has positive
  dimension, then $\Gamma_3(X_u,X^v) = X$, see \Lemma{cayleyG3}.  It
  follows that
  \[
  [\cO_{\Gamma_3(X_u,X^v)}] =
  \begin{cases}
    \cO^{s_\ga} & \text{if $\dim(X_u)=\dim(X^v)=0$\,;} \\
    1 & \text{otherwise.}
  \end{cases}
  \]
  This is compatible with the quantum equals classical theorem for the
  Cayley plane proved in \cite{chaput.perrin:rationality}, which
  implies that $I_3(\cO_u, \cO^v, \cO_w^\vee)$ is equal to
  $\delta_{w,s_\ga}$ when $\dim(X_u) = \dim(X^v) = 0$, and equal to
  $\delta_{w,1}$ otherwise.
\end{proof}

Let $u, v \in W^P$ and define the varieties
$\Bl_d(X_u,X^v) = \pi^{-1}(M_d(X_u,X^v))$ and
$Z_d^{(3)}(X_u,X^v) = (p e_1)^{-1}(X_u) \cap (p e_2)^{-1}(X^v)$.  Then
\eqn{diagram} restricts to the following commutative diagram.  It
follows from \cite[Prop.~3.2 and Thm.~2.5]{buch.chaput.ea:finiteness}
that all varieties in this diagram are irreducible and have rational
singularities.

\begin{equation}\label{eqn:restrict}
  \xymatrix{\Bl_d(X_u,X^v) \ar[rr]^{\pi} \ar[d]^{\phi} & &
    M_d(X_u,X^v) \ar[d]^{\ev_3} \\
    Z_d^{(3)}(X_u,X^v) \ar[r]^{e_3} & Z_d(X_u,X^v) \ar[r]^{p}
    & \Gamma_d(X_u,X^v)}
\end{equation}

It follows from Kleiman's transversality theorem
\cite{kleiman:transversality} that the restricted maps $\pi$ and
$\phi$ of \eqn{restrict} have the same properties as the corresponding
maps of the diagram \eqn{diagram}, i.e.\ $\pi$ is birational and the
general fibers of $\phi$ are rational.  \Proposition{gysin} therefore
implies that $\pi$ and $\phi$ are cohomologically trivial, and
\Proposition{projrich} shows that $p$ is cohomologically trivial.
\Theorem{main}(b) therefore follows from \Lemma{composite} and
\Proposition{gysin} together with the following statement.

\begin{prop}
  The general fibers of the restricted map
  $e_3 : Z_d^{(3)}(X_u,X^v) \to Z_d(X_u,X^v)$ are rational.
\end{prop}
\begin{proof}
  For any $\omega \in Y_d$ the variety
  $X_\omega \cong q^{-1}(\omega) \subset Z_d$ is a homogeneous space.
  Furthermore, for any Schubert variety $\Omega \subset X$ we have
  $\Omega \cap X_\omega \cong p^{-1}(\Omega) \cap q^{-1}(\omega)$.  It
  therefore follows from \Proposition{projschub} that
  $\Omega \cap X_\omega$ is a Schubert variety for all points $\omega$
  in a dense open subset of $q(p^{-1}(\Omega)) \subset Y_d$.  Since
  the fiber of $e_3 : Z_d^{(3)} \to Z_d$ over an arbitrary point
  $(\omega,x_3) \in Z_d$ is given by
  $e_3^{-1}(\omega,x_3) \cong X_\omega \times X_\omega$, we obtain
  $Z_d^{(3)}(X_u,X^v) \cap e_3^{-1}(\omega,x_3) \cong (X_u \cap
  X_\omega) \times (X^v \cap X_\omega)$.  This proves the proposition.
\end{proof}

\begin{remark}\label{remark:codim}
  The codimension of $\Gamma_d(X_u,X^v)$ can be computed as follows.
  We assume that $d$ is a well behaved degree and $u,v \in W^P$.  Let
  $Q \subset G$ be a parabolic subgroup containing $B$ such that
  $Y_d = G/Q$.  Let $\ov{u}$ be the maximal element in the coset
  $u w_P W_Q$ where $w_P$ is the longest element in $W_P$, and let
  $\ov{v}$ be the minimal element in $v W_Q$.  Set $F = G/B$ and let
  $\rho : F \to X$ be the projection.  It follows from
  \eqn{gwprojrich} that
  $\Gamma_d(X_u,X^v) = \rho(F_{\ov{u}} \cap F^{\ov{v}})$, so we have
  $\Gamma_d(X_u,X^v) \neq \emptyset$ if and only if
  $\ov{v} \leq \ov{u}$ in the Bruhat order on $W$.  The {\em Hecke
    product\/} on $W$ is the unique associative monoid product such
  that $s_i \cdot w$ is equal to $s_i w$ if $\ell(s_i w) > \ell(w)$
  and equal to $w$ otherwise (see \cite[\S3]{buch.mihalcea:curve} for
  details and references).  Choose $a \in W_P$ such that $\ov{v}a$ is
  the maximal element in $\ov{v} W_P$.  It follows from
  \cite[Prop.~3.3]{knutson.lam.ea:projections} that
  $\rho(F_{\ov{u}} \cap F^{\ov{v}}) = \rho(F_{\ov{u}\cdot a} \cap
  F^{\ov{v}a})$
  where $\ov{u} \cdot a$ is the Hecke product of $\ov{u}$ and $a$.
  Since $\rho$ maps $F^{\ov{v}a}$ birationally onto its image in $X$,
  we deduce that $\Gamma_d(X_u,X^v)$ is birational to
  $F_{\ov{u}\cdot a} \cap F^{\ov{v}a}$.  This implies that
  $\dim \Gamma_d(X_u,X^v) = \ell(\ov{u}\cdot a) - \ell(\ov{v}a)$.
\end{remark}

%
%
\begin{example}
  Let $G = \GL(6)$ and set $F = \Fl(6) = G/B$ where $B$ is the Borel
  subgroup of upper triangular matrices.  The Weyl group of $G$ is the
  symmetric group $S_6$, where each permutation $w \in S_6$ is
  identified with the permutation matrix whose entry in position
  $(i,j)$ is equal to 1 whenever $i = w(j)$.  Let
  $\cO^w = [\cO_{F^w}] \in K(F)$ be the Grothendieck class of the
  opposite Schubert variety $F^w = \ov{B^- w.B} \subset F$.  Consider
  the Richardson variety $R = F_u \cap F^v$ in $F$ defined by the
  permutations $u = 642153$ and $v = 132546$ in $S_6$.  A computation
  with Grothendieck polynomials
  \cite{lascoux.schutzenberger:structure} then gives
  \[
  [\cO_R] = \cO^{w_0 u} \cdot \cO^v = \cO^{154623} +2\cO^{245613}
  +\cO^{253614} -\cO^{345612} -3\cO^{254613} +\cO^{354612}
  \]
  in $K(F)$ where $w_0 = 654321$ is the longest permutation.

  Now let $P$ be the parabolic subgroup such that
  $B \subset P \subset G$ and $X = G/P = \Gr(2,6)$ is the Grassmann
  variety of 2-planes in $\C^6$.  Let $\rho : F \to X$ be the
  projection.  The set $W^P$ consists of permutations $w$ for which
  $w(j)<w(j+1)$ for $j \neq 2$, and each such permutation $w$ can be
  identified with the partition $(w(2)\!-\!2\,,\,w(1)\!-\!1)$.  With
  this notation we obtain from \Proposition{projrich} that the class
  of the projected Richardson variety $\rho(R)$ is given by
  \[
  [\cO_{\rho(R)}] = \rho_*([\cO_R]) = \cO^{(3,0)} +2\cO^{(2,1)}
  -2\cO^{(3,1)} -\cO^{(2,2)} +\cO^{(3,2)}
  \]
  in $K(X)$.  The Grothendieck classes of all projected Gromov-Witten
  varieties $\Gamma_d(X_\la,X^\mu)$ in $X$ can be computed using
  \Corollary{pgw_class} combined with \Theorem{qclas} or the Pieri
  rule for $\QK(X)$ obtained in
  \cite[Thm.~5.4]{buch.mihalcea:quantum}.  It turns out that the class
  $[\cO_{\rho(R)}]$ is not the class of any projected Gromov-Witten
  variety in $X$.  Therefore not all projected Richardson varieties in
  $X$ are projected Gromov-Witten varieties.
\end{example}

\section{The Cayley plane}\label{sec:cayley}

The `quantum equals classical' theorem proved in
\cite{chaput.perrin:rationality} for Gromov-Witten invariants of
degree $3$ of the Cayley plane $E_6/P_6$ involves a variety $Z_3$ that
is not a homogeneous space.  It is therefore not possible to use
\Proposition{projrich} to prove that the map $p$ of the diagram
\eqn{restrict} is cohomologically trivial.  In this section we give a
different proof of \Theorem{main}(b) for the Cayley plane when $d=3$.

Let $X$ be a cominuscule variety and fix a degree $d \geq 0$.  Set
$\cZ_d = \cZ_{d,3} = \ev(M_d) \subset X^3$.  Given three subvarieties
$\Omega_1$, $\Omega_2$, $\Omega_3$ of $X$ we define the varieties
$M_d(\Omega_1,\Omega_2,\Omega_3) = \ev^{-1}(\Omega_1 \times \Omega_2
\times \Omega_3) \subset M_d$,
$\cZ_d(\Omega_1,\Omega_2,\Omega_3) = \cZ_d \cap (\Omega_1 \times
\Omega_2 \times \Omega_3)$,
and
$\cZ_d(\Omega_1,\Omega_2) = \cZ_d \cap (\Omega_1 \times \Omega_2
\times X)$.
It was proved in \cite{buch.mihalcea:quantum,
  chaput.perrin:rationality} that the general fibers of the map
$\ev : M_d \to \cZ_d$ are rational.  Let $U \subset \cZ_d$ be a dense
open subset such that $M_d(x,y,z) = \ev^{-1}(x,y,z)$ is rational for
all points $(x,y,z) \in U$.  Let $G$ act diagonally on $X^3$ and on
$\cZ_d$.  Since the evaluation map is equivariant for this action, we
may assume that $U$ is $G$-stable.

\begin{lemma}\label{lemma:meetU}
  Let $u, v \in W^P$ satisfy $\Gamma_d(X_u,X^v) \neq \emptyset$.  Then
  $\cZ_d(X_u,X^v) \cap U \neq \emptyset$.
\end{lemma}
\begin{proof}
  Set $d_0 = \min(d,d_X(2))$.  The assumption implies that we may
  choose $(x_0,y_0) \in X_u \times X^v$ such that $d(x_0,y_0) = d_0$.
  Furthermore, since $\cZ_d$ contains a dense open subset of points
  $(x,y,z)$ for which $d(x,y)=d_0$, we may choose $(x,y,z) \in U$ with
  $d(x,y) = d_0$.  Finally, since $G$ acts transitively on the set of
  pairs of points of distance $d_0$ by
  \cite[Prop.~18]{chaput.manivel.ea:quantum*1}, we have
  $(x_0,y_0) = g.(x,y)$ for some $g \in G$.  It follows that
  $(x_0,y_0,g.z) = g.(x,y,z) \in \cZ_d(X_u,X^v) \cap U$, as required.
\end{proof}

Now let $X = E_6/P_6$ be the Cayley plane.  We need the following
fact.

\begin{lemma}\label{lemma:cayleyG3}
  Let $X$ be the Cayley plane and let $\Omega_1$ and $\Omega_2$ be
  irreducible closed subvarieties of $X$.  If both of these varieties
  are single points and $d(\Omega_1,\Omega_2) = 2$, then
  $\Gamma_3(\Omega_1, \Omega_2)$ is a translate of the Schubert
  divisor $X^{s_\ga}$.  In all other cases we have
  $\Gamma_3(\Omega_1,\Omega_2) = X$.
\end{lemma}
\begin{proof}
  If $\Omega_1$ and $\Omega_2$ are single points, then the lemma
  follows from \cite[Cor.~4.6]{buch.chaput.ea:finiteness}.  Assume
  that $\Omega_2$ has positive dimension and let $x \in \Omega_1$ be
  any point.  It is enough to show that $\Gamma_3(x,\Omega_2) = X$.
  If $z \in X$ is any point, then $\Gamma_3(x,z)$ contains a divisor,
  and this implies that $\Gamma_3(x,z) \cap \Omega_2 \neq \emptyset$.
  It follows that $z \in \Gamma_3(x,\Omega_2)$.
\end{proof}

\Theorem{main}(b) for the Cayley plane and $d=3$ is a consequence of
the following result together with \Proposition{gysin}.

\begin{prop}
  Let $X$ be the Cayley plane and let $u,v \in W^P$.  Then the general
  fibers of the evaluation map
  $\ev_3 : M_3(X_u,X^v) \to \Gamma_3(X_u,X^v)$ are rationally
  connected.
\end{prop}
\begin{proof}
  For convenience we let $x_0 = 1.P$ denote the $B$-fixed point and
  $y_0 = w_0.P$ the $B^-$-fixed point in $X$.  The evaluation map
  $\ev_3 : M_3(X_u,X^v) \to \Gamma_3(X_u,X^v)$ can be factored as
  \[
  M_3(X_u,X^v) \xrightarrow{\,\ev\,} \cZ_3(X_u,X^v)
  \xrightarrow{\,p_3\,} \Gamma_3(X_u,X^v)
  \]
  where $p_3$ is the restriction of the third projection $X^3 \to X$.
  If we have $\dim(X_u) = \dim(X^v) = 0$, then $X_u = \{x_0\}$,
  $X^v = \{y_0\}$, and the map $p_3$ is an isomorphism.  The
  proposition therefore follows from \Lemma{meetU}.

  Assume now that $\dim(X^v) \geq 1$.  Choose $z_0 \in X$ such that
  $d(x_0,z_0) = 2$ and $\Gamma_2(x_0,z_0)$ is a $B$-stable Schubert
  variety.  It then follows from
  \cite[Cor.~4.6]{buch.chaput.ea:finiteness} that $\Gamma_3(x_0,z_0)$
  is the $B$-stable Schubert divisor.  Since $G.(x_0,z_0)$ is a dense
  open subset of $X^2$ and $B^- B$ is a dense open subset of $G$, it
  follows that $O = B^-B.(\{x_0\} \times X \times \{z_0\})$ is a dense
  open subset of $X^3$.  We claim that
  $O \cap U \cap \cZ_3(X_u,X^v) \neq \emptyset$.  To see this, note at
  first that $\cZ_3(X_u,X^v) = \ev(M_3(X_u,X^v))$ is irreducible, so
  $U \cap \cZ_3(X_u,X^v)$ is a dense open subset by \Lemma{meetU}.  In
  addition, if $y' \in \Gamma_3(x_0,z_0) \cap X^v$ is any point, then
  $(x_0,y',z_0) \in O \cap \cZ_3(X_u,X^v)$, so $O \cap \cZ_3(X_u,X^v)$
  is also a dense open subset of $\cZ_3(X_u,X^v)$.  This proves the
  claim.  Since $p_3 : \cZ_3(X_u,X^v) \to X$ is surjective, it follows
  from Kleiman's transversality theorem \cite{kleiman:transversality}
  that there exists a dense open subset $V \subset X$ such that, for
  each $z \in V$ the fiber $p_3^{-1}(z) = \cZ_3(X_u,X^v,z)$ is locally
  irreducible and each component of this fiber meets $O \cap U$.

  Fix a point $z \in V$ and consider the first projection
  $p_1 : \cZ_3(X_u,X^v,z) \to X_u$.  For any point $x \in X_u$ we have
  $p_1^{-1}(x) = \cZ_3(x,X^v,z) \cong \Gamma_3(x,z) \cap X^v$.  Since
  this variety is connected by \Lemma{cayleyG3} and
  \Proposition{degrich}, we deduce that $\cZ_3(X_u,X^v,z)$ is
  connected and hence irreducible.  By definition of $V$ we have
  $\cZ_3(X_u,X^v,z) \cap O \neq \emptyset$.  If $(x,y,z)$ is any point
  in this intersection, then we may write $(x,z) = b'b.(x_0,z_0)$
  where $b' \in B^-$ and $b \in B$, after which we obtain
  $p_1^{-1}(x) \cong b'b.\Gamma_3(x_0,z_0) \cap X^v =
  b'.(\Gamma_3(x_0,z_0) \cap X^v)$.
  This shows that the general fibers of the map
  $p_1 : \cZ_3(X_u,X^v,z) \to X_u$ are translates of Richardson
  varieties, so it follows from \Proposition{ratconn} that
  $\cZ_3(X_u,X^v,z)$ is rationally connected.  The definition of $V$
  also implies that $\cZ_3(X_u,X^v,z) \cap U \neq \emptyset$.
  Therefore the general fibers of
  $\ev : M_3(X_u,X^v,z) \to \cZ_3(X_u,X^v,z)$ are rationally
  connected, so a second application of \Proposition{ratconn} shows
  that $M_3(X_u,X^v,z)$ is rationally connected for all $z \in V$.
  This finishes the proof.
\end{proof}

\section{The square of a point}\label{sec:square}

In this section we apply our results to compute the square of a point
in the (small) quantum $K$-theory ring $\QK(X)$ of any cominuscule
variety $X$.  The ring $\QK(X)$ is a formal deformation of the
Grothendieck ring $K(X)$, which as a group is defined by
$\QK(X) = K(X) \otimes_\Z \Z \llbracket q \rrbracket$.  The product
$\cO^u \star \cO^v$ of two Schubert structure sheaves in $\QK(X)$ is
given by
\[
\cO^u \star \cO^v = \sum_{w,d\geq 0} N^{w,d}_{u,v}\, q^d\, \cO^w
\]
where the sum is over all $w \in W^P$ and non-negative degrees $d$.
The structure constants $N^{w,d}_{u,v}$ are defined by the recursive
identity
\[
N^{w,d}_{u,v} \ = \ I_d(\cO^u, \cO^v, \cO_w^\vee) - \sum_{\kappa,e>0}
N^{\kappa,d-e}_{u,v} \cdot I_e(\cO^\kappa, \cO_w^\vee) \,,
\]
where the sum is over all $\kappa \in W^P$ and non-zero degrees $e$
with $0 < e \leq d$.  The degree zero constants $N^{w,0}_{u,v}$ are
the structure constants of the Grothendieck ring $K(X)$.  The quantum
$K$-theory ring $\QK(X)$ was defined by Givental \cite{givental:wdvv}
and was further studied in connection with the finite difference Toda
lattice \cite{givental.lee:quantum, braverman.finkelberg:finite}.
While the product $\cO^u \star \cO^v$ is defined as a power series in
$q$ and might {\em a priori\/} contain infinitely many non-zero terms,
it was proved in \cite{buch.chaput.ea:finiteness} that
$N^{w,d}_{u,v} = 0$ whenever $X$ is cominuscule and $d > d_X(2)$.  A
finiteness result for the quantum $K$-theory of a larger class of
homogeneous spaces was obtained in \cite{buch.chaput.ea:rational}.

\newcommand{\cOp}{\cO_{\mathrm{pt}}}
\newcommand{\pp}{\mathrm{pt}}
\begin{cor}
  Let $X$ be a cominuscule variety and let $\cOp \in K(X)$ denote the
  class of a point.  Choose $\kappa \in W^P$ such that
  $\Gamma_{d_X(2)}(1.P, w_0.P)$ is a translate of $X^\kappa$.  Then we
  have $\cOp \star \cOp = q^{d_X(2)}\, \cO^\kappa$ in $\QK(X)$.
\end{cor}
\begin{proof}
  Recall that $\Gamma_{d_X(2)}(1.P, w_0.P)$ is a Schubert variety in
  $X$ by \cite[Prop.~18]{chaput.manivel.ea:quantum*1}.  Let
  $N^{w,d}_{\pp,\pp}$ denote the structure constants defining the
  product $\cOp \star \cOp$.  It follows from \Corollary{pgw_class}
  that
  $I_d(\cOp, \cOp, \cO_w^\vee) = \euler{X}([\cO_{\Gamma_d(1.P,
    w_0.P)}] \cdot \cO_w^\vee)$
  for all degrees $d$ and $w \in W^P$.  Since
  $\Gamma_d(1.P,w_0.P) = \emptyset$ for $d < d_X(2)$, it follows by
  induction on $d$ that $N^{w,d}_{\pp,\pp} = 0$ for $d < d_X(2)$.
  This in turn implies that
  $N^{w,d_X(2)}_{\pp,\pp} = \euler{X}([\cO_{\Gamma_{d_X(2)}(1.P,
    w_0.P)}] \cdot \cO_w^\vee) = \euler{X}(\cO^\kappa \cdot
  \cO_w^\vee) = \delta_{\kappa,w}$.
  The corollary follows from this together with the fact
  \cite{buch.chaput.ea:finiteness} that $N^{w,d}_{\pp,\pp} = 0$ for
  $d > d_X(2)$.
\end{proof}


\begin{thebibliography}{10}

\bibitem{anderson.griffeth.ea:positivity}
D.~Anderson, S.~Griffeth, and E.~Miller, \emph{Positivity and {K}leiman
  transversality in equivariant {$K$}-theory of homogeneous spaces}, J. Eur.
  Math. Soc. (JEMS) \textbf{13} (2011), no.~1, 57--84. \MR{2735076}

\bibitem{billey.coskun:singularities}
S.~Billey and I.~Coskun, \emph{Singularities of generalized {R}ichardson
  varieties}, Comm. Algebra \textbf{40} (2012), no.~4, 1466--1495. \MR{2912998}

\bibitem{braverman.finkelberg:finite}
A.~Braverman and M.~Finkelberg, \emph{Finite difference quantum {T}oda lattice
  via equivariant {$K$}-theory}, Transform. Groups \textbf{10} (2005), no.~3-4,
  363--386. \MR{2183117 (2007d:17019)}

\bibitem{brion:positivity}
M.~Brion, \emph{Positivity in the {G}rothendieck group of complex flag
  varieties}, J. Algebra \textbf{258} (2002), no.~1, 137--159, Special issue in
  celebration of Claudio Procesi's 60th birthday. \MR{1958901 (2003m:14017)}

\bibitem{buch:littlewood-richardson}
A.~S. Buch, \emph{A {L}ittlewood-{R}ichardson rule for the {$K$}-theory of
  {G}rassmannians}, Acta Math. \textbf{189} (2002), no.~1, 37--78. \MR{1946917}

\bibitem{buch:quantum}
\bysame, \emph{Quantum cohomology of {G}rassmannians}, Compositio Math.
  \textbf{137} (2003), no.~2, 227--235. \MR{1985005}

\bibitem{buch.chaput.ea:finiteness}
A.~S. Buch, P.-E. Chaput, L.~Mihalcea, and N.~Perrin, \emph{Finiteness of
  cominuscule quantum {$K$}-theory}, Ann. Sci. \'Ec. Norm. Sup\'er. (4)
  \textbf{46} (2013), no.~3, 477--494 (2013). \MR{3099983}

\bibitem{buch.chaput.ea:rational}
\bysame, \emph{Rational connectedness implies finiteness of quantum
  {$K$}-theory}, Asian J. Math. \textbf{20} (2016), no.~1, 117--122.
  \MR{3460760}

\bibitem{buch.kresch.ea:gromov-witten}
A.~S. Buch, A.~Kresch, and H.~Tamvakis, \emph{Gromov-{W}itten invariants on
  {G}rassmannians}, J. Amer. Math. Soc. \textbf{16} (2003), no.~4, 901--915.
  \MR{1992829}

\bibitem{buch.mihalcea:quantum}
A.~S. Buch and L.~Mihalcea, \emph{Quantum {$K$}-theory of {G}rassmannians},
  Duke Math. J. \textbf{156} (2011), no.~3, 501--538. \MR{2772069}

\bibitem{buch.mihalcea:curve}
\bysame, \emph{Curve neighborhoods of {S}chubert varieties}, J. Differential
  Geom. \textbf{99} (2015), no.~2, 255--283. \MR{3302040}

\bibitem{chaput.manivel.ea:quantum*1}
P.-E. Chaput, L.~Manivel, and N.~Perrin, \emph{Quantum cohomology of minuscule
  homogeneous spaces}, Transform. Groups \textbf{13} (2008), no.~1, 47--89.
  \MR{2421317 (2009e:14095)}

\bibitem{chaput.perrin:rationality}
P.-E. Chaput and N.~Perrin, \emph{Rationality of some {G}romov-{W}itten
  varieties and application to quantum {$K$}-theory}, Commun. Contemp. Math.
  \textbf{13} (2011), no.~1, 67--90. \MR{2772579 (2012k:14077)}

\bibitem{fulton.pandharipande:notes}
W.~Fulton and R.~Pandharipande, \emph{Notes on stable maps and quantum
  cohomology}, Algebraic geometry---{S}anta {C}ruz 1995, Proc. Sympos. Pure
  Math., vol.~62, Amer. Math. Soc., Providence, RI, 1997, pp.~45--96.
  \MR{1492534 (98m:14025)}

\bibitem{givental:wdvv}
A.~Givental, \emph{On the {WDVV} equation in quantum {$K$}-theory}, Michigan
  Math. J. \textbf{48} (2000), 295--304, Dedicated to William Fulton on the
  occasion of his 60th birthday. \MR{1786492 (2001m:14078)}

\bibitem{givental.lee:quantum}
A.~Givental and Y.-P. Lee, \emph{Quantum {$K$}-theory on flag manifolds,
  finite-difference {T}oda lattices and quantum groups}, Invent. Math.
  \textbf{151} (2003), no.~1, 193--219. \MR{1943747 (2004g:14063)}

\bibitem{graber.harris.ea:families}
T.~Graber, J.~Harris, and J.~Starr, \emph{Families of rationally connected
  varieties}, J. Amer. Math. Soc. \textbf{16} (2003), no.~1, 57--67
  (electronic). \MR{1937199 (2003m:14081)}

\bibitem{he.lam:projected}
X.~He and T.~Lam, \emph{Projected {R}ichardson varieties and affine {S}chubert
  varieties}, arXiv:1106.2586.

\bibitem{kleiman:transversality}
S.~L. Kleiman, \emph{The transversality of a general translate}, Compositio
  Math. \textbf{28} (1974), 287--297. \MR{0360616 (50 \#13063)}

\bibitem{knutson.lam.ea:positroid}
A.~Knutson, T.~Lam, and D.~Speyer, \emph{Positroid varieties: juggling and
  geometry}, Compos. Math. \textbf{149} (2013), no.~10, 1710--1752.
  \MR{3123307}

\bibitem{knutson.lam.ea:projections}
\bysame, \emph{Projections of {R}ichardson varieties}, J. Reine Angew. Math.
  \textbf{687} (2014), 133--157. \MR{3176610}

\bibitem{kollar:higher}
J.~Koll{\'a}r, \emph{Higher direct images of dualizing sheaves. {I}}, Ann. of
  Math. (2) \textbf{123} (1986), no.~1, 11--42. \MR{825838 (87c:14038)}

\bibitem{lascoux.schutzenberger:structure}
A.~Lascoux and M.-P. Sch{\"u}tzenberger, \emph{Structure de {H}opf de l'anneau
  de cohomologie et de l'anneau de {G}rothendieck d'une vari\'et\'e de
  drapeaux}, C. R. Acad. Sci. Paris S\'er. I Math. \textbf{295} (1982), no.~11,
  629--633. \MR{686357 (84b:14030)}

\bibitem{lusztig:total}
G.~Lusztig, \emph{Total positivity in reductive groups}, Lie theory and
  geometry, Progr. Math., vol. 123, Birkh\"auser Boston, Boston, MA, 1994,
  pp.~531--568. \MR{1327548 (96m:20071)}

\bibitem{richardson:intersections}
R.~W. Richardson, \emph{Intersections of double cosets in algebraic groups},
  Indag. Math. (N.S.) \textbf{3} (1992), no.~1, 69--77. \MR{1157520
  (93b:20081)}

\bibitem{rietsch:closure}
K.~Rietsch, \emph{Closure relations for totally nonnegative cells in {$G/P$}},
  Math. Res. Lett. \textbf{13} (2006), no.~5-6, 775--786. \MR{2280774
  (2007j:14073)}

\bibitem{sierra:general}
S.~J. Sierra, \emph{A general homological {K}leiman-{B}ertini theorem}, Algebra
  Number Theory \textbf{3} (2009), no.~5, 597--609. \MR{2578891 (2011d:14088)}

\bibitem{zak:tangents}
F.~L. Zak, \emph{Tangents and secants of algebraic varieties}, Translations of
  Mathematical Monographs, vol. 127, American Mathematical Society, Providence,
  RI, 1993, Translated from the Russian manuscript by the author. \MR{1234494
  (94i:14053)}

\end{thebibliography}

\providecommand{\bysame}{\leavevmode\hbox to3em{\hrulefill}\thinspace}
\providecommand{\MR}{\relax\ifhmode\unskip\space\fi MR }
\providecommand{\MRhref}[2]{%
  \href{http://www.ams.org/mathscinet-getitem?mr=#1}{#2}
}
\providecommand{\href}[2]{#2}

\end{document}